\pdfoutput=1
\documentclass[11pt]{article}

\usepackage[T1]{fontenc}
\usepackage[utf8]{inputenc}
\IfFileExists{lmodern.sty}{\usepackage{lmodern}}{}
\usepackage{amsmath,amssymb,amsthm,mathtools}
\IfFileExists{microtype.sty}{\usepackage[final,expansion=false]{microtype}}{}
\usepackage[letterpaper,margin=1in]{geometry}
\usepackage{xcolor}
\usepackage{hyperref}

\hypersetup{
  colorlinks=true,
  linkcolor=blue!50!black,
  citecolor=blue!50!black,
  urlcolor=blue!50!black,
  pdftitle={Large Sets of Integers with No Harmonic Triples},
  pdfauthor={Samuel Korsky}
}

\allowdisplaybreaks

\theoremstyle{plain}
\newtheorem{theorem}{Theorem}[section]
\newtheorem{lemma}[theorem]{Lemma}

\theoremstyle{remark}
\newtheorem{question}[theorem]{Question}

\newcommand{\F}{\mathbb F}
\newcommand{\AP}{\mathrm{AP}}

\title{Large Sets of Integers with No Harmonic Triples}
\author{Samuel Korsky}
\date{July 5, 2026}

\begin{document}

\maketitle

\begin{abstract}
\noindent
Let $f(N)$ denote the largest size of a set $A\subseteq [N]=\{1,\ldots,N\}$ containing no distinct $a,b,c$ such that
\[
  \frac2a=\frac1b+\frac1c .
\]
We prove
\[
  f(N)\gg N\exp\!\left(-(2\sqrt{\log(24/7)}+o(1))\sqrt{\log\log N}\right).
\]
The construction filters the odd integers up to $N$ by a random affine image of a dense three-term-progression-free set in a prime field $\F_q$ with $q\asymp\log N$, and then deletes a controlled family of collapsed triples.
\end{abstract}

\section{Introduction}

A triple of positive integers is called \emph{harmonic} if the reciprocals of its entries form a three-term arithmetic progression.  In this note we study large subsets of $[N]=\{1,\ldots,N\}$ containing no nontrivial harmonic triple.  More precisely, let $f(N)$ be the maximum cardinality of a set $A\subseteq [N]$ such that there do not exist distinct $a,b,c\in A$ with
\begin{equation}\label{eq:harmonic-equation}
  \frac2a=\frac1b+\frac1c .
\end{equation}

The problem is a reciprocal analogue of the classical question of estimating $r_3(N)$, the largest size of a subset of $[N]$ with no nontrivial three-term arithmetic progression.  The latter problem was introduced by Erd\H{o}s and Tur\'an \cite{ErdosTuran1936} and has been central in additive combinatorics since Behrend's construction \cite{Behrend1946} and Roth's theorem \cite{Roth1953}.  The lower-bound side was developed by Salem and Spencer \cite{SalemSpencer1942}, Behrend \cite{Behrend1946}, Elkin \cite{Elkin2011}, Green and Wolf \cite{GreenWolf2010}, and most recently Elsholtz, Hunter, Proske and Sauermann \cite{EHPS2024}.  The upper-bound side has also seen major progress, including the work of Bloom and Sisask \cite{BloomSisask2020}, Kelley and Meka \cite{KelleyMeka2023}, and Bloom and Sisask's subsequent improvement of the Kelley--Meka bound \cite{BloomSisask2023}.  These upper bounds do not directly apply to \eqref{eq:harmonic-equation}, since the reciprocal map is not an additive map on intervals of integers.

Throughout, all logarithms are natural.  We use the following lower bound of Elsholtz, Hunter, Proske and Sauermann \cite{EHPS2024}.  If $r_3(X)$ denotes the maximum size of a subset of $\{1,\ldots,X\}$ with no nontrivial three-term arithmetic progression, then
\begin{equation}\label{eq:EHPS-input}
  r_3(X)\ge X\exp\!\left(-(C_{\AP}+o(1))\sqrt{\log X}\right),
  \qquad C_{\AP}=2\sqrt{\log(24/7)} .
\end{equation}
The constant in \eqref{eq:EHPS-input} is the natural-logarithm form of the base-two constant appearing in \cite{EHPS2024}.

Our main result transfers this improved Behrend-type lower bound to the harmonic-triple-free setting, with $\log N$ replaced by $\log\log N$.

\begin{theorem}\label{thm:main}
As $N\to\infty$,
\begin{equation}\label{eq:main-bound}
  f(N)\gg N\exp\!\left(-(C_{\AP}+o(1))\sqrt{\log\log N}\right).
\end{equation}
\end{theorem}

The proof is direct.  Choose a prime $q\asymp\log N$ and a dense three-term-progression-free set $R\subseteq\F_q$.  A random affine image of $R$ is used to select those odd $n\le N$ for which $n^{-1}\bmod q$ lies in that affine image.  A harmonic triple surviving this filter either gives a nonconstant three-term progression in the affine image, which is impossible, or collapses modulo $q$.  A parametrization of odd harmonic triples shows that the collapsed triples have a difference parameter divisible by $q$, and there are only $O(N\log N/q)$ such triples.  Each collapsed triple survives the random filter with probability only $\delta=|R|/q$.  Since the initial random set has expected size $\gg \delta N$, choosing the implicit constant in $q\asymp\log N$ large enough leaves $\gg\delta N$ elements after deletion.

The complementary upper-bound problem remains open here.  It is natural to ask whether positive density is impossible.

\begin{question}\label{ques:density-zero}
Is it true that
\[
  f(N)=o(N)
\]
as $N\to\infty$?
\end{question}

\section{Preliminaries}

We write $X\ll Y$ and $Y\gg X$ to mean $|X|\le C Y$ for an absolute constant $C$.  Constants implicit in $O(\cdot)$, $\ll$ and $\gg$ are absolute unless explicitly stated otherwise.

A subset $R$ of an abelian group of odd order is called \emph{three-term-progression-free} if there are no $x,y,z\in R$, not all equal, such that $x+z=2y$.

\begin{lemma}[Progression-free subsets of prime fields]\label{lem:finite-field-input}
Let $q$ be prime and let $q\to\infty$.  There is a three-term-progression-free set $R\subseteq\F_q$ such that
\[
  \frac{|R|}{q}\ge \exp\!\left(-(C_{\AP}+o(1))\sqrt{\log q}\right),
  \qquad C_{\AP}=2\sqrt{\log(24/7)} .
\]
\end{lemma}

\begin{proof}
Let $M=\lfloor q/3\rfloor$.  By \eqref{eq:EHPS-input}, there is a set $S\subseteq\{1,\ldots,M\}$ with no nontrivial three-term arithmetic progression and
\[
  |S|\ge M\exp\!\left(-(C_{\AP}+o(1))\sqrt{\log M}\right).
\]
Embed $S$ into $\F_q$ in the natural way and call the image $R$.  This embedding creates no wrap-around progressions.  Indeed, if $x,z,y\in S$ and $x+z\equiv 2y\pmod q$, then
\[
  |x+z-2y|<q,
\]
so $x+z=2y$ over the integers.  Hence $x=y=z$.  Finally, $M/q\ge 1/4$ for all sufficiently large $q$, and $\log M=\log q+O(1)$.  Multiplying by the fixed factor $M/q$ changes the exponent by only $O(1)$, which is $o(\sqrt{\log q})$.  This proves the asserted density bound.
\end{proof}

We shall also use the following elementary parametrization.

\begin{lemma}[Parametrization of odd harmonic triples]\label{lem:parametrization}
Let $b<a<c$ be odd positive integers satisfying
\begin{equation}\label{eq:param-equation}
  \frac2a=\frac1b+\frac1c .
\end{equation}
Then there exist positive integers $m,h,d$ with $1\le d<h$ such that
\begin{equation}\label{eq:parametrization}
  b=(h-d)hm,\qquad a=(h-d)(h+d)m,\qquad c=(h+d)hm .
\end{equation}
Conversely, every choice of positive integers $m,h,d$ with $1\le d<h$ gives a positive integer solution of \eqref{eq:param-equation} through \eqref{eq:parametrization}; it need not be an odd solution.  Moreover, if $q$ is an odd prime, $q\nmid abc$, and
\begin{equation}\label{eq:equal-congruence-param}
  a\equiv b\equiv c\pmod q,
\end{equation}
then $q\mid d$.
\end{lemma}

\begin{proof}
Write $b=gx$ and $c=gy$, where $(x,y)=1$ and $x<y$.  Since $b$ and $c$ are odd, both $x$ and $y$ are odd.  From $2bc=a(b+c)$ we obtain
\[
  2gxy=a(x+y).
\]
Because $(xy,x+y)=1$, the integer $x+y$ divides $2g$.  Thus $2g=m(x+y)$ for some integer $m\ge 1$, and then $a=mxy$.  Set
\[
  h=\frac{x+y}{2},\qquad d=\frac{y-x}{2}.
\]
Then $h$ and $d$ are positive integers with $d<h$, and \eqref{eq:parametrization} follows from $x=h-d$, $y=h+d$, and $g=hm$.

The converse is verified by direct substitution:
\[
  \frac1{(h-d)hm}+\frac1{(h+d)hm}
  =\frac{2}{(h-d)(h+d)m}.
\]
Finally assume \eqref{eq:equal-congruence-param}.  Since $q\nmid abc$, each factor $h-d$, $h$, $h+d$ and $m$ is nonzero modulo $q$.  Comparing $a$ and $b$ modulo $q$ in \eqref{eq:parametrization} and cancelling the common nonzero factor $(h-d)m$ gives
\[
  h+d\equiv h\pmod q,
\]
so $q\mid d$.
\end{proof}

\section{The Construction}

The following lemma is the main construction.

\begin{lemma}[Cardinality construction]\label{lem:cardinality-core}
Let $Y$ be large.  There is a set $B\subseteq [Y]$ containing no distinct $a,b,c$ satisfying \eqref{eq:harmonic-equation} and such that
\begin{equation}\label{eq:core-cardinality}
  |B|
  \gg
  Y\exp\!\left(-(C_{\AP}+o(1))\sqrt{\log\log Y}\right).
\end{equation}
\end{lemma}

\begin{proof}
Let $K\ge 1$ be a sufficiently large absolute constant, to be fixed at the end of the proof.  By Bertrand's postulate, for all sufficiently large $Y$ there is a prime $q$ with
\begin{equation}\label{eq:q-choice}
  K\log Y\le q\le 2K\log Y .
\end{equation}
In particular, $q$ is odd and $\log q=\log\log Y+O(1)$.

Let $R\subseteq\F_q$ be supplied by Lemma~\ref{lem:finite-field-input}, and put
\[
  \delta=\frac{|R|}{q}.
\]
Thus
\begin{equation}\label{eq:delta-bound}
  \delta\ge \exp\!\left(-(C_{\AP}+o(1))\sqrt{\log\log Y}\right).
\end{equation}
For $\lambda\in\F_q^\times$ and $\mu\in\F_q$, write
\[
  R_{\lambda,\mu}=\lambda R+\mu .
\]
These affine images remain three-term-progression-free.

Let
\[
  V=\{n\le Y:n\text{ is odd and }q\nmid n\}.
\]
Then $|V|\gg Y$.  Choose $\lambda\in\F_q^\times$ and $\mu\in\F_q$ uniformly at random, and define
\[
  B_0=B_0(\lambda,\mu)=\{n\in V:n^{-1}\bmod q\in R_{\lambda,\mu}\}.
\]
For each fixed nonzero residue $r\in\F_q$, the probability that $r\in R_{\lambda,\mu}$ is exactly $\delta$.  Hence
\begin{equation}\label{eq:initial-expected-size}
  \mathbb E\!\left[|B_0|\right]=\delta |V|\gg \delta Y .
\end{equation}

We next delete the largest element of every surviving harmonic triple. For a fixed affine image $R_{\lambda,\mu}$, delete from $B_0$ every element $c$ that is the largest member of some triple $b<a<c$ in $B_0$ satisfying \eqref{eq:harmonic-equation}.  Call the remaining set $B=B(\lambda,\mu)$.  This set is harmonic-triple-free, since the largest element of any surviving harmonic triple would have been deleted.

It remains to bound the expected number of deletions.  Suppose $b<a<c$ is a harmonic triple in $B_0$.  Reducing \eqref{eq:harmonic-equation} modulo $q$ gives
\[
  2a^{-1}\equiv b^{-1}+c^{-1}\pmod q.
\]
The three residues $b^{-1},a^{-1},c^{-1}$ lie in the progression-free set $R_{\lambda,\mu}$, so they must all be equal.  Therefore
\begin{equation}\label{eq:abc-congruent}
  a\equiv b\equiv c\pmod q .
\end{equation}
Since $b,a,c\in V$, they are odd and not divisible by $q$.  By Lemma~\ref{lem:parametrization}, every such triple has the form \eqref{eq:parametrization} with $q\mid d$.

Let
\[
  U_q(Y)=
  \#\{(h,d,m):1\le d<h,\ q\mid d,\ h(h+d)m\le Y\}.
\]
This is an overcount for the possible triples whose largest element can be deleted.  If a single element $c$ is deleted because of several triples, then counting all of those triples only increases the upper bound.  We have
\begin{align}
  U_q(Y)
  &\le
  \sum_{h\le Y^{1/2}}
  \sum_{\substack{1\le d<h\\ q\mid d}}
  \left\lfloor \frac{Y}{h(h+d)}\right\rfloor \notag\\
  &\le
  Y\sum_{h\le Y^{1/2}}
  \sum_{1\le j<h/q}
  \frac1{h(h+qj)} \notag\\
  &\ll
  Y\sum_{h\le Y^{1/2}}\frac1{qh}
  \ll \frac{Y\log Y}{q}.\label{eq:Uq-bound}
\end{align}
For each parametrized triple counted by $U_q(Y)$, the probability that it contributes to a deletion is at most $\delta$.  Indeed, if the displayed integers are not all odd, or if one of them is divisible by $q$, then this probability is zero.  Otherwise $q\mid d$ implies that the three entries are congruent to the same nonzero residue modulo $q$, so their inverse residues are also equal; the probability that this common inverse residue lies in $R_{\lambda,\mu}$ is exactly $\delta$.  By the union bound, the expected number of deleted elements is at most $\delta U_q(Y)$.

Combining \eqref{eq:initial-expected-size}, \eqref{eq:Uq-bound}, and \eqref{eq:q-choice}, we get
\[
  \mathbb E\!\left[|B|\right]
  \ge
  \delta\left(c_1Y-c_2\cdot\frac{Y\log Y}{q}\right)
  \ge
  \delta\left(c_1Y-\frac{c_2}{K}\cdot Y\right)
\]
for absolute constants $c_1,c_2>0$.  Taking $K>2c_2/c_1$, the expectation is $\gg \delta Y$.  Therefore some affine image yields a set $B$ satisfying \eqref{eq:core-cardinality}, by \eqref{eq:delta-bound}.
\end{proof}

\begin{proof}[Proof of Theorem~\ref{thm:main}]
Apply Lemma~\ref{lem:cardinality-core} with $Y=N$.  The resulting set is a subset of $[N]$, contains no distinct solution of \eqref{eq:harmonic-equation}, and satisfies \eqref{eq:main-bound}.  Hence $f(N)\ge |B|$ gives the theorem.
\end{proof}

\end{document}